\documentclass[11pt]{article}

\usepackage{geometry}
\usepackage{tikz}
\usepackage{cite}
\usepackage[utf8]{inputenc}
\usepackage[shortlabels]{enumitem}
\usepackage{amsmath,amsfonts,amssymb,amsthm}
\usepackage{comment}
\usepackage[hidelinks]{hyperref}

\theoremstyle{plain}
\newtheorem{theorem}{Theorem}[section]
\newtheorem{lemma}[theorem]{Lemma}
\newtheorem{proposition}[theorem]{Proposition}
\newtheorem{corollary}[theorem]{Corollary}

\theoremstyle{definition}
\newtheorem{definition}[theorem]{Definition}
\newtheorem{example}[theorem]{Example}

\theoremstyle{remark}
\newtheorem{remark}{Remark}[section]

\newcommand{\Z}{\mathbb{Z}}

\newcommand{\R}{\mathbb{R}}
\newcommand{\C}{\mathbb{C}}

\renewcommand{\i}{\mathrm{i}}
\newcommand{\g}{\mathfrak{g}}
\renewcommand{\t}{\mathfrak{t}}
\newcommand{\Aut}{\mathrm{Aut}}
\newcommand{\End}{\mathrm{End}}
\newcommand{\Hom}{\mathrm{Hom}}
\newcommand{\Spin}{\mathrm{Spin}}
\newcommand{\Spinc}{\mathrm{Spin\textsuperscript{c}}}
\newcommand{\Ad}{\mathrm{Ad}}
\newcommand{\vol}{\mathrm{vol}}

\renewcommand{\H}{\mathcal{H}}
\renewcommand{\L}{\mathcal{L}}
\newcommand{\M}{\mathcal{M}}
\renewcommand{\O}{\mathcal{O}}
\newcommand{\X}{\mathfrak{X}}

\renewcommand{\subset}{\subseteq}

\renewcommand{\d}{\mathrm{d}}

\title{Quantization of Polysymplectic Manifolds}
\author{Casey Blacker}
\date{}

\begin{document}\maketitle

\begin{abstract}
	We adapt the framework of geometric quantization to the polysymplectic setting. Considering prequantization as the extension of symmetries from an underlying polysymplectic manifold to the space of sections of a Hermitian vector bundle, a natural definition of prequantum vector bundle is obtained which incorporates in an essential way the action of the space of coefficients. We define quantization with respect to a polarization and to a spin\textsuperscript{c} structure. In the presence of a complex polarization, it is shown that the polysymplectic Guillemin-Sternberg conjecture is false. We conclude with potential extensions and applications.
\end{abstract}

\let\thefootnote\relax
\footnote{\hspace{-10pt}\textit{Date.} July 23, 2019}
\footnote{\hspace{-10pt}\textit{2010 Mathematics Subject Classification.} 53D05, 53D50, 53D20, 53C27.}
\footnote{\hspace{-10pt}\textit{Key words and phrases.} Polysymplectic manifolds, geometric quantization, moment maps, Dirac operators.}
\footnote{\hspace{-10pt}\textit{E-mail}. \texttt{cblacker@math.ecnu.edu.cn}}

\tableofcontents


\section{Introduction}

Geometric quantization encompasses a range of techniques which, broadly speaking, take a symplectic manifold and return a complex Hilbert space. While these methods were inspired by the relation between classical and quantum physics, they are by no means constrained by this interpretation and enjoy a range of applications from representation theory \cite{Kirillov04} to computational geometry \cite{Vergne96b}. A polysymplectic manifold is a smooth manifold equipped with a closed and nondegenerate $2$-form taking values in a fixed vector space. The aim of this paper is to adapt the methods of geometric quantization to the setting of polysymplectic manifolds.

The original aim of polysymplectic geometry was to furnish a general mathematical framework for relativistic field theory \cite{Gunther87,BinzSniatyckiFischer88,CantrijnIbortLeon99}. In this setting, the traditional Hamiltonian formalism exhibits two defects. First, it requires a continuum degrees of freedom; second, it enforces an artificial distinction between space and time. In response to these complication, the physical \emph{de Donder-Weyl formalism} \cite{Donder35,Weyl35} was developed, and \emph{multisymplectic geometry} arose as the corresponding mathematical framework. See \cite[Subsection 1.8]{Helein12} for a historical discussion. Motivated by the same considerations, G\"unther \cite{Gunther87,Gunther87a} later introduced the polysymplectic formalism as an alternative to the multisymplectic approach. Accommodating a wide degree of variation, both approaches remain current in their application to relativistic field theory today.

Beyond physics, the idea of a vector-valued symplectic structure, or its equivalent, has arisen independently on multiple occasions in the literature. Most prominent is the \emph{$k$-symplectic formalism} of Awane \cite{Awane92}, according to which a $k$-symplectic structure on a smooth manifold $M$, chosen so that $k+1$ divides $\dim M$, is defined to be a collection of $k$ presymplectic structures $(\omega_1,\ldots,\omega_k)$, which are collectively nondegenerate in the sense that $\cap_i \ker\omega_i$ vanishes as a distribution on $TM$. Similar ideas have appeared in the theories of \emph{$k$-almost cotangent structures} \cite{LeonMendezSalgado93}, \emph{generalized symplectic geometry} \cite{Norris93} and \emph{$n$-symplectic geometry} \cite{Norris01}, and have found applications in, for example, Lie group thermodynamics \cite{Barbaresco17}.

While the quantization of polysymplectic field theories has been the subject of extensive research, see for example \cite{Kanatchikov04,Kanatchikov98a,Kanatchikov01,Kanatchikov15,Bashkirov04,Sardanashvily02,BashkirovSardanashvily04}, the problem of quantizing general polysymplectic manifolds has received little attention. In this regard, Awane and Goze \cite{AwaneGoze00} have introduced a notion of $k$-symplectic geometric prequantization. We compare their construction with our own in Section \ref{sec:prequantization}.

We now outline the contents of this paper.

In Section \ref{sec:quantization_review} we review the basic constructions of geometric quantization in such a way as to clarify our later approach in the polysymplectic context. Specifically, we consider geometric quantization as a means of extending the symmetries of a symplectic manifold $(M,\omega)$ to the space of sections a Hermitian line bundle $L\to M$, subject to a constraint that we identify with the Heisenberg uncertainty principle.

In Section \ref{sec:polysymplectic_review} we review polysymplectic geometry from the $V$-symplectic standpoint. We introduce the technical condition of \emph{transitivity}, later to play an essential role in establishing the connection between prequantizations and prequantum vector bundles. We also introduce the notion of a \emph{local polysymplectic structure}, that is, a polysymplectic structure with local coefficients, and compare it with the more familiar \emph{multisymplectic structure}.

Section \ref{sec:prequantization} is the heart of the paper. Working by analogy with the symplectic context, we define a \emph{prequantization} to be an admissible extension of the Hamiltonian dynamics on an underlying $V$-symplectic manifold $(M,\omega)$ to a Hermitian vector bundle $E\to M$. We note that this approach is similar to the original perspective of \cite{Souriau70}. In the case that $(M,\omega)$ is transitive and connected, there is a natural equivalence between this construction and that of a prequantum vector bundle, defined below.

\setcounter{section}{4}
\setcounter{theorem}{5}
\begin{definition}
	A \emph{prequantum vector bundle} on $(M,\omega)$ consists of a faithful Hermitian $V$-module bundle $E\to M$ with a $V$-linear unitary connection $\nabla$ satisfying $F^\nabla = -\omega$.
\end{definition}

A comparison of the symplectic and polysymplectic situations is given as follows. We make use of the fact that the space of coefficients $V$ inherits the structure of an abelian Lie algebra when identified with the space of constant Hamiltonian functions on $(M,\omega)$.

\begin{center}
\begin{tabular}{ll}
	symplectic											&$V$-symplectic												\\ \hline
														&															\\ [-.35cm]
	prequantum line bundle $(L,\nabla)$					&prequantum vector bundle $(E,\nabla,A)$					\\
	scalar multiplication $m$							&effective unitary representation $A:V\to\End\,E$			\\
	Planck's constant $\hbar>0$							&weights of $A$												\\
	prequantum operator $Q_f=\nabla_{X_f}\hspace{-2pt}+\i\hbar\hspace{1pt}m_f$
														&$Q_f=\nabla_{X_f} + A_f$									\\[1pt]
	curvature condition $F^\nabla = -\i\hbar\hspace{1pt}\hspace{1pt}\omega$
														&$F^\nabla = -A_\omega$
\end{tabular}
\end{center}

We obtain a very strong result on the global structure of prequantum vector bundles.

\setcounter{section}{4}
\setcounter{theorem}{7}
\begin{theorem}
	Every prequantum vector bundle $(E,\nabla,A)$ splits as the sum of weight bundles $\oplus_{\lambda\in w(A)}(E_\lambda,\nabla|_{E_\lambda},\lambda)$.
\end{theorem}

Given the minimal nature of our initial construction of a prequantization, it is notable that we obtain such a strong result. A slightly weaker version holds in the local polysymplectic setting.

\setcounter{section}{4}
\setcounter{theorem}{22}
\begin{theorem}
	If $(M,\omega)$ is a transitive and connected local $V$-symplectic manifold with prequantum vector bundle $(E,\nabla,A)$, then the action of the holonomy group $\mathrm{Hol}_x^\nabla$ preserves the weight-space decomposition $E_x=\oplus_{\lambda\in w(A_x)} (E_x)_\lambda$ up to permutation of the factors.
\end{theorem}

We define a \emph{prequantum lattice} associated to $(M,\omega)$ to be a full sublattice $I\subset V$ such that $\omega$ lies in the image of $H^2(M,I)\hookrightarrow H^2(M,V)$. In terms of this construction, we obtain the following result on the existence and classification of prequantum vector bundles.

\setcounter{section}{4}
\setcounter{theorem}{12}
\begin{theorem}
	If $(M,\omega)$ is transitive and connected, then there is a bijection between equivalence classes of minimal rank prequantizations $(E,\nabla,A)$ on $(M,\omega)$ and bases $\mathcal{B}=\frac{1}{2\pi\i}w(A)$ of prequantum lattices $I\subset V$.
\end{theorem}

In the symplectic setting, the discreteness of the image of the natural pairing $\langle\omega,\cdot\,\rangle:H_2(M,\Z)\to \R$ is the source of the term \emph{quantum} \cite[V.III.A.2]{Dugas88}. In light of this, we may take Theorem \ref{thm:fundamental_theorem_of_prequantum_lattices} as further evidence for the aptness of our use of this terminology.

In Section \ref{sec:polarized_quantization} we incorporate \emph{polarizations} into our quantization scheme. As in the symplectic case, a polarization is defined to be an integrable Lagrangian distribution of $T^\C M$. Examples of polysymplectic manifolds quantized with respect to a polarization include the following.

\begin{center}
\begin{tabular}{llll}
	classical states&prequantum states				&polarization							&quantum states		\\ \hline
					&								&										&					\\ [-.35cm]
	$G$				&$\X^\C(G)$						&translates of max.\ torus $T\subset G$	&$\X^\C(G)_T$		\\ [.04cm]
	$\Hom(TQ,V)$	&$C^\infty\big(\Hom(TQ,V),V^\C\big)$	&vertical foliation of $\Hom(TQ,V)$		&$C^\infty(Q,V^\C)$
\end{tabular}
\end{center}
Here $G$ is a compact semisimple Lie group and $Q$ is a smooth manifold, as reviewed in Section \ref{sec:polysymplectic_review}. By $\X^\C(G)_T$ we denote the space of $T$-invariant complex vector fields on $G$. The main result of this section is that the Guillemin-Sternberg conjecture does not obtain in the polysymplectic context.

\setcounter{section}{5}
\setcounter{theorem}{15}
\begin{theorem}
	Let $(E,\nabla,A)$ be a positive definite prequantum vector bundle on $(M,\omega,J,G,\mu)$ and suppose that $(M_0,\omega_0)$ is nonempty and $V$-symplectic, and inherits a complex structure $J_0$ and prequantum vector bundle $E_0$. It is not generally the case that $\H_J(M)_G\cong\H_{J_0}(M_0)$.
\end{theorem}

In Section \ref{sec:spinc_quantization} we observe that the formalism of spin\textsuperscript{c} quantization admits a natural extension to the local polysymplectic setting. In particular, we define the spin\textsuperscript{c} quantization of a polysymplectic manifold, with respect to a given prequantum vector bundle, to be the index space of an associated spin\textsuperscript{c} Dirac operator.

We conclude in Section \ref{sec:outlook} with a discussion of potential interactions with Chern-Simons theory, multisymplectic quantization, and quantum field theory.

\section*{Acknowledgements}
The author would like to thank his advisor, Xianzhe Dai, and to acknowledge the support of the East China Normal University and the China Postdoctoral Science Foundation.

\setcounter{section}{1}
\setcounter{theorem}{0}


\section{Review of Geometric Quantization}\label{sec:quantization_review}

We first review geometric quantization in the symplectic setting. In this section, we consider quantization with respect to a polarization; spin\textsuperscript{c} quantization is reserved for Section \ref{sec:spinc_quantization}. This formalism is attributed by Weinstein \cite{Weinstein81} to Souriau \cite{Souriau70} and Kostant \cite{Kostant73}. We refer to \cite{EcheverriaEnriquezMunozLecanda98,Woodhouse92,Lerman12} for modern treatments, and to \cite{BatesWeinstein97} for a discussion of alternative approaches.


\subsection{Prequantization}

Symplectic geometry arose as a framework for classical mechanics. In modern language, the possible states of a given mechanical system form the points of a \emph{configuration manifold} $Q$. The space $TQ$ of infinitesimal state transitions is called the \emph{(generalized) velocity phase space}. Conceptually, the momentum associated to a velocity $X\in T_qQ$ is a measure of the sensitivity of the kinetic energy $k:TQ\to\R$ to variations in $X$. Formally, we identify the associated momentum with the linear map $Y\mapsto\frac{\d}{\d t}\,k(X+tY)\,|_{t=0}$ on $T_qQ$. This establishes a bundle morphism $TQ\to T^*Q$, the \emph{Legendre transformation} associated to $k$, which identifies $T^*Q$ as the \emph{(generalized) momentum phase space} of the underlying mechanical system. It is often convenient to work with the cotangent bundle $T^*Q$ since it possesses a canonical symplectic structure in terms of which the mechanical symmetries of the underlying system are easily expressed. See \cite{Marsden92} for a review of classical mechanics from this perspective.

An application of Darboux's theorem thus provides that every symplectic manifold $(M,\omega)$ is locally equivalent to the momentum phase space of a mechanical system. The unattainable goal of \emph{geometric quantization} is to assign to $(M,\omega)$ a unique corresponding \emph{space of quantum states} $\mathcal{H}$, in such a manner that encodes the relation between physical classical systems and the quantum systems from which they arise. In particular, we aim to adapt $(M,\omega)$ to incorporate what we may heuristically associate with the following three physical properties:

\begin{enumerate}[1.]
	\item \textbf{Superposition}: Quantum states may be \emph{superposed}, that is, combined additively.

	\item \textbf{Phase}: Quantum states exhibit the presence of an unphysical periodic degree of freedom, known as \emph{phase}, associated to the states of the corresponding classical system. Under superposition, quantum states of identical phase combine constructively, while those of opposite phase combine destructively.
	
	\item \textbf{Uncertainty}: Quantum states may not be simultaneously characterized in terms of a pair of conjugate coordinates. If $(x_i,y_i)_{i\leq n}$ is a system of symplectic coordinates on the classical phase space, and if we associate to the quantum state $\psi$ a definite value of $x_i$, then we may not associate to $\psi$ \emph{any} value of the conjugate coordinate $y_i$.
\end{enumerate}

With this in mind, we define $\mathcal{H}$ to be the space of smooth sections of a Hermitian line bundle $L\to M$. Note that we do not impose the usual $L^2$ condition on the members of $\mathcal{H}$. Superposition of quantum states corresponds to addition of sections, and pointwise phase differences between quantum states are naturally described by means of the angular coordinate on the complex fibers of $L$.

In \emph{prequantization}, we postpone considerations of uncertainty. Inspired by physical systems, we associate to each classical observable $f\in C^\infty(M)$ an infinitesimal symmetry $Q_f\in\End\,\mathcal{H}$ extending the classical symmetry $X_f\in\X(M)$ in such a way that the action of $f$ on the phase of $\psi\in\mathcal{H}$ is proportional at each point to the value of $f$ by a fixed constant $\hbar>0$. That is, we define $Q_f = \nabla_{X_f}+\i\hbar\hspace{.5pt}f$, where $\nabla$ is a connection on $L$. See the beginning of Subsection \ref{subsec:construction_of_prequantizations} for a discussion of our use of the term \emph{extension}. We remark that our definition of $Q_f$ differs from the usual conventions of the physics literature by a factor of $\i\hbar$. Adopting the requirement that $Q:C^\infty(M)\to \End\,\H$ be a map of Lie algebras, a necessary and sufficient condition for the existence of $(L,\nabla)\to M$, and hence of $Q$, is that $F^\nabla = -\i\hbar\hspace{1pt}\omega$.

Observe that, while we have described an assignment $Q$ of quantum \emph{symmetries}, we have not proposed any identification between classical and quantum \emph{states}. Indeed, any such effort must take care to not violate uncertainty principle, and leads to the consideration of \emph{coherent states}. However, in light of the uncertainty principle, it is the Lagrangian submanifolds of $(M,\omega)$, known in this context as \emph{semiclassical states}, which are more readily associated with the elements of $\H$. These considerations, however, will not play a role in this paper.


\subsection{Polarized Quantization}

To incorporate the uncertainty principle, we consider only those quantum states $\psi\in\mathcal{H}$ which are parallel along a \emph{polarization} $\mathcal{\mathcal{P}}$, that is, an integrable Lagrangian distribution of the complexified tangent bundle $T^\C M$. Such states $\psi\in\mathcal{H}_{\mathcal{P}}$ are said to be \emph{polarized}. We note that while the uncertainty principle applies to \emph{every} system of symplectic coordinates $(x_i,y_i)_{i\leq n}$, our choice of polarization $\mathcal{\mathcal{P}}$ is fixed. As a partial solution, we note that in various situations it is possible to canonically identify $\H_{\mathcal{P}}$ and $\H_{\mathcal{P}'}$ for distinct $\mathcal{P}$ and $\mathcal{P}'$ \cite{Woodhouse92}.

There are multiple ways to address the fact that the symmetries $Q_f$ do not generally preserve $\mathcal{H}_{\mathcal{P}}$. We will take the mathematically convenient approach and simply restrict the representation $Q:C^\infty(M)\to\End\,\mathcal{H}$ to the subalgebra of classical observables $C^\infty(M)_{\mathcal{P}}$ which do preserve $\mathcal{H}_{\mathcal{P}}$. The \emph{quantization} of $(M,\omega)$ with respect to $\mathcal{\mathcal{P}}$ is defined to be the restricted representation $Q:C^\infty(M)_{\mathcal{P}}\to\End\,\mathcal{H}_{\mathcal{P}}$.

If $M$ is a K\"ahler manifold, then the antiholomorphic tangent bundle $T^{0,1}M\subset T^\C M$ forms a polarization on $M$, and the identity $\omega=-\i\hbar\,F^\nabla$ implies that $L\to M$ is a holomorphic line bundle. In this case, the \emph{K\"ahler quantization} of $M$ is defined to be the quantization of $M$ with respect to $\mathcal{P}=T^{0,1}M$. In other words, $\mathcal{H}_{\mathcal{P}}$ is the space of holomorphic sections of $L\to M$. Note that if $M$ is compact, then $\mathcal{H}_{\mathcal{P}}$ is a finite dimensional Hilbert space.


\section{The Polysymplectic Formalism}\label{sec:polysymplectic_review}

This section serves two purposes. First, we provide a review of the fundamental aspects of polysymplectic geometry. Second, we introduce for the first time certain ``classical'' notions which we will utilize in the quantum material to follow.


\subsection{$V$-Hamiltonian Systems}

Let us briefly review the $V$-Hamiltonian formalism. The following conventions, terminology, and notation are broadly consistent with the more comprehensive presentation in our earlier work \cite{Blacker19}.

\begin{definition}
	Let $M$ be a smooth manifold and $V$ a real vector space. For simplicity, we will assume that $V$ is finite-dimensional. A \emph{$V$-symplectic structure} on $M$ is a closed, nondegenerate $2$-form $\omega$ with values in $V$. The \emph{Hamiltonian vector field} associated to a function $f\in C^\infty(M,V)$ is the unique vector $X_f\in\X(M)$ satisfying $d f=-\iota_{X_f}\omega$. In this case, we call $f$ a \emph{Hamiltonian function} associated to $X_f$. We denote the vector space of Hamiltonian functions by $C_H^\infty(M,V)$. Finally, the \emph{bracket} of two observables $f,h\in C_H^\infty(M,V)$ is defined to be $\{f,h\}= \omega(X_f,X_h) = X_f h\in C_H^\infty(M,V)$. We define the \emph{component} of $\omega$ with respect to the dual coefficient $\lambda\in V^*$ to be the real-valued $2$-form $\omega_\lambda = \lambda\circ\omega$.
\end{definition}

Note that our definition of the bracket differs by a factor of $-1$ from our convention in \cite{Blacker19}.

In contrast with the symplectic situation, it is not the case that every $V$-valued function possesses an associated Hamiltonian vector field. Indeed, the condition for $f\in C^\infty(M,V)$ to be a Hamiltonian function is that $\d f\in \Omega^1(M,V)$ lies in the image of the map $\iota\omega:\X(M)\to\Omega^1(M,V)$ given by $X\mapsto \iota_X\omega$.

\begin{example}
	\begin{enumerate}[i.]
		\item If $\theta\in\Omega^1(G,\g)$ is the Cartan $1$-form on the centerless Lie group $G$, then $\omega=-\d\theta$ is a $\g$-symplectic structure on $G$. The local model is $(\g,[\,,])$, where the Lie bracket $[\,,]$ represents a linear $\g$-symplectic structure on $\g$.
		\item Consider a smooth manifold $Q$ and a vector space $V$. The fundamental $1$-form $\theta\in\Omega^1\big(\Hom(TQ,V),V\big)$ on $\pi:\Hom(TQ,V)\to Q$ is given by $\theta(X) = \phi(\pi_*X)$ for $X\in T_\phi\Hom(TQ,V)$, and the canonical $V$-symplectic structure on $\Hom(TQ,V)$ is $\omega=-\d\theta$. The local model consists of the vector space $U\oplus\Hom(U,V)$ and the product $(u+\phi,u'+\phi')\mapsto \phi'(u)-\phi(u')$, where $U$ is the linear space on which $Q$ is modeled.
	\end{enumerate}
\end{example}

\begin{definition}
	Let $(M,\omega)$ be a $V$-symplectic manifold equipped with the action of a Lie group $G$. We will define a \emph{comoment map} for the action of $G$ to be a Lie algebra antihomomorphism $\tilde\mu:\g\to C_H^\infty(M,V)$ such that $\tilde\mu(\xi) = X_{\underline\xi}$, where $\underline\xi\in\X(M)$ is the fundamental vector field of $\xi\in\g$. Explicitly, $\underline\xi_x\mapsto\frac{\d}{\d t}e^{t\xi}x\big|_{t=0}$ for every $x\in M$. The associated \emph{moment map} $\mu:M\to\Hom(\g,V)$, where $\Hom(\g,V)$ denotes the space of linear maps from $\g$ to $V$, is defined by the formula $\mu(x)(\xi) = \tilde\mu(\xi)(x)$, for all $x\in M$ and $\xi\in\g$. The tuple $(M,\omega,G,\mu)$ is called a $V$-\emph{Hamiltonian system}. The \emph{reduction} of $(M,\omega,G,\mu)$ at $0\in\Hom(\g,V)$ is defined to be the pair $(M_0,\omega_0)$ consisting of the set $M_0=\mu^{-1}(0)/G$ and the unique $V$-valued $2$-form $\omega_0$ on the regular part of $M_0$ that satisfies $i^*\omega = \pi^*\omega_0$, where $i:\mu^{-1}(0)\to M$ is the inclusion and $\pi:\mu^{-1}(0)\to M_0$ the projection.
\end{definition}

The polysymplectic reduction theorem, which guarantees the existence and uniqueness of the reduced $2$-form $\omega_0$, was first proved in \cite{MarreroRoman-RoySalgadoVilarino15} and constitutes the polysymplectic analogue of the celebrated symplectic reduction theorem of Marsden and Weinstein \cite{MarsdenWeinstein74}. We refer to \cite{Blacker19} for a statement and proof in the $V$-symplectic setting.

The comoment map $\tilde\mu$ is a lift of the assignment $f\mapsto X_f$ of Hamiltonian vector fields. That is,
\begin{center}
	\begin{tikzpicture}[scale=2.3]
		\node (ul) at (1,.8) {$C^\infty(M)$};
		\node (ll) at (0,0) {$\g$};
		\node (lr) at (1,0) {$\X(M)$};
		
		\path[->] (ul) edge node[right] {$X$} (lr);
		\path[->,dashed] (ll) edge node[above left] {$\tilde\mu$} (ul);
		\path[->] (ll) edge node[below] {$\lambda_*$} (lr);
	\end{tikzpicture}
\end{center}
where $\lambda_*$ denotes the assignment $\xi\mapsto\underline\xi$ of fundamental vector fields. With our conventions, $X$ is a Lie algebra homomorphism, while $\tilde\mu$ and $\lambda_*$ are antihomomorphisms.

\begin{example}
	\begin{enumerate}[i.]
		\item The left regular action of $G$ on $(G,-\d\theta)$ is Hamiltonian with moment map $\Ad:G\to\End(\g)$. In particular, for each $\xi\in\g$, the function $g\mapsto \Ad_g\,\xi$ is Hamiltonian, with associated Hamiltonian vector field the right-invariant vector field $\underline{\xi}\in\X(G)$.
		\item A right action of $G$ on $Q$ induces a Hamiltonian left action of $G$ on $\Hom(TQ,V)$ with moment map given by $\mu(\phi)(\xi)=\phi(\underline\xi_Q)$. If the quotient $Q/G$ is a manifold, then the reduction of $\Hom(TQ,V)$ at $0\in\Hom(\g,V)$ is naturally isomorphic to $\Hom(T(Q/G),V)$ with its canonical symplectic structure.
	\end{enumerate}
\end{example}


\subsection{Dynamics and Invariant Measures}

Let $(M,\omega)$ be a $V$-symplectic manifold.

\begin{definition}
	We say that $(M,\omega)$ is \emph{transitive at $x\in M$} if every tangent vector $X_x\in T_xM$ extends to a Hamiltonian vector field $X\in\X(M)$. We say that $(M,\omega)$ is \emph{transitive} if it is transitive at every point of $M$.
\end{definition}

If $(M,\omega)$ is transitive and connected, then for any two points $x,y\in M$ there is a time dependent Hamiltonian vector field $X_t\in \R\times\X(M)$ with flow $\phi_t$ satisfying $\phi_1(x)=y$. See \cite{RyvkinWurzbacher15} for a similar condition in the multisymplectic setting.

\begin{definition}
	An \emph{invariant measure} on $(M,\omega)$ is a section of the density bundle $\eta\in |\Lambda|(M)$ which is preserved by every Hamiltonian vector field on $M$.
\end{definition}

When $M^n$ is orientable, we often identify $\eta$ with a representative from $\Omega^n(M)$.

\begin{lemma}\label{lem:transitive_uniqueness_of_invariant_measures}
	If $(M^n,\omega)$ is a transitive, connected $V$-symplectic manifold, and if $\eta\in|\Lambda|(M)$ is a nonzero invariant measure, then $\eta$ is nonvanishing and unique up to rescaling.
\end{lemma}

\begin{proof}
	Suppose for a contradiction that $\eta$ vanishes at some point $x\in M$. Transitivity and invariance imply that $\eta=0$ on a neighborhood of $x$, and consequently that the vanishing set of $\eta$ is open and nontrivial. Since $M$ is connected and the vanishing set is closed, we obtain the desired contradiction that $\eta=0$. Therefore, $\eta$ is nonvanishing.

	We now prove uniqueness up to rescaling. Since $\eta$ is nonvanishing, it follows that that every smooth measure on $M$ is of the form $s\eta$ for some positive $s\in C^\infty(M)$. If $s\eta$ is invariant, then
	\[
		0 = \L_X (s\eta) = (Xs)\eta,\hspace{.8cm}X\in\X_H(M),
	\]
	from which obtain $Xs=0$ for every local Hamiltonian vector field $X$. Since $(M,\omega)$ is transitive and connected, we conclude that $s$ is constant.
\end{proof}

In contrast to the symplectic case, not every polysymplectic manifold admits a nonzero invariant measure.

\begin{example}\label{eg:transitive_dynamics}
	\begin{enumerate}[i.]
		\item Suppose the manifold $Q$ is modeled on the vector space $U$. Every constant vector field on $U\oplus\Hom(U,V)$ preserves the constant $V$-symplectic structure $\omega(u+\phi,u'+\phi') = \phi'(u)-\phi(u')$, and is thus locally Hamiltonian. Since every coordinate chart $\phi:O\subset Q\to U$ determines a local $V$-symplectomorphism from $U\oplus\Hom(U,V)$ to $\Hom(TQ,V)$, it follows that $\Hom(TQ,V)$ is transitive.
		
			The space $U\oplus\Hom(U,V)$ possesses a nonzero invariant measure only if every linear $V$-symplectomorphism has determinant $\pm1$. However, for $\alpha>1$, the $V$-symplectomorphism $(u+\phi)\mapsto(\alpha u+\alpha^{-1}\phi)$ has determinant $0<\alpha^{k(1-\ell)}<1$. Therefore, $U\oplus\Hom(U,V)$ does not admit a nonzero invariant measure.
		\item Consider $S^2$ as the unit sphere in $\R^3$ and let $\bar\omega\in\Omega^2(S^2,\R)$ be any symplectic structure. Define the $\R^3$-symplectic form $\omega\in\Omega^2(S^2,\R^3)$ by $\omega(X,Y) = \bar\omega(X,Y)\cdot p$ for $X,Y\in T_pS^2$. Since the only nontrivial $\R^3$-symplectomorphism of $(S^2,\omega)$ is reflection through the origin, it follows that $\omega$ is not preserved by any nontrivial local vector field. Therefore, $(S^2,\omega)$ is nowhere transitive and every measure on $S^2$ is an invariant measure.
	\end{enumerate}
\end{example}

We circumvent this difficulty by restricting the space of admissible Hamiltonian vector fields to those associated to a distinguished class of \emph{classical observables}, defined as follows.

\begin{definition}
	An \emph{algebra of (classical) observables} $\O$ is any Lie subalgebra of $C_H^\infty(M,V)$.
\end{definition}

Throughout this paper, we will assume that $\O$ contains the constant functions $V\subset C_H^\infty(M,\omega)$.

We say that $(M,V)$ is transitive with respect to $\O$, or that $\eta\in|\Lambda|(M)$ is invariant with respect to $\O$, when the indicated condition obtains with respect to the image of $\O$ under the map $X:C_H^\infty(M,V)\to\X(M)$.

\begin{example}
	Let $\O\subset C_H^\infty(G,\g)$ consist of those functions of the form $g\mapsto \Ad_g\,\xi$ for $\xi\in\g$. Since the right-invariant vector fields $\underline\g\subset\X(G)$ are Hamiltonian, it follows that $(G,-\d\theta)$ is transitive with respect to $\O$, and that the $\O$-invariant measures on $G$ are precisely the left Haar measures.
\end{example}


\subsection{Local Polysymplectic Manifolds}

We now describe a natural extension of the $V$-symplectic formalism to the setting of local coefficients.

\begin{definition}
	A \emph{local $V$-symplectic manifold} consists of a smooth manifold $M$, a flat vector bundle $(\tilde{V},\nabla)\to M$ modeled on $V$, and a $\d^\nabla$-closed nondegenerate $2$-form $\omega\in\Omega^2(M,\tilde{V})$. A \emph{Hamiltonian section} $f\in\Gamma(M,\tilde{V})$ is one for which there exists a \emph{Hamiltonian vector field} $X\in\X(M)$ with $\d^\nabla f = -\iota_X\omega$.
\end{definition}

By trivializing $\tilde{V}\to M$ over a neighborhood $O\subset M$, a local $V$-Hamiltonian manifold is identified with to a $V$-Hamiltonian manifold on $O$ in the natural way. It is interesting to compare this construction with the multisymplectic formalism.

\begin{definition}
	A \emph{multisymplectic structure} on $M$ is a closed nondegenerate differential form $\Omega\in\Omega^*(M)$. More specifically, $\Omega$ is said to be \emph{$k$-plectic} if $\Omega\in\Omega^{k+1}(M)$. A \emph{Hamiltonian form} $\alpha\in\Omega^{k-1}(M)$ is one for which there exists a \emph{Hamiltonian vector field} $X\in\X(M)$ with $\d\alpha = -\iota_X\Omega$.
\end{definition}

\begin{remark}
	In fact, the multisymplectic formalism is occasionally more general than what we have just defined. We refer to \cite{RyvkinWurzbacher19,CantrijnIbortLeon99} for relevant introductions.
\end{remark}

If $\Lambda^{k-1}T^*M$ admits a flat connection, then the map
\[
	\bar{\cdot}:\Omega^\ell(M)\to\Omega^2(M,\Lambda^{\ell-2} T^*M)
\]
given by
\[
	\bar\alpha(X,Y) = \iota_Y\iota_X\alpha \in\Omega^{\ell-2}(M),	\hspace{1cm}X,Y\in\X(M),
\]
takes a $k$-plectic structure $\Omega\in\Omega^{k+1}(M)$ to a local polysymplectic structure $\bar\Omega\in\Omega^2(M,\Lambda^{k-1}T^*M)$. However, this transformation does not preserve Hamiltonian dynamics. In particular, it is not true that $\d\alpha=\iota_X\Omega$ implies $\d^\nabla\bar\alpha = -\iota_X\bar\Omega$, for $\alpha\in\Omega^{k-1}(M)$ and $X\in\X(M)$. We note that the assignment $\Omega\mapsto\bar\Omega$ is the \emph{symbol map} of \cite{ForgerGomes13}.


\section{Prequantization}\label{sec:prequantization}

Our aim in this section is to extend the theory of symplectic prequantization to the polysymplectic setting. Specifically, we aim to extend the Hamiltonian symmetries of an algebra of classical observables $\O\subset C_H^\infty(M,V)$ to an algebra of symmetries of the sections of a Hermitian vector bundle $E\to M$, in such a way that the nonzero values of an observable $f\in\O$ at the points of $M$ generate nontrivial unitary transformations of the corresponding fibers of $E$. When $(M,\omega,\O,\eta)$ is transitive, and $\eta$ is a nonvanishing invariant measure, we arrive at an equivalent theory of prequantum vector bundles. After investigating the consequences of these definitions, including classification schemes and conditions for existence, we touch on some differences that are encountered in the local polysymplectic setting.


\subsection{Construction of Prequantizations}\label{subsec:construction_of_prequantizations}

Fix a $V$-symplectic manifold $(M,\omega)$, an algebra of classical observables $\O\subset C_H^\infty(M,V)$, and an invariant measure $\eta$ on $M$.

\begin{definition}
	A \emph{prequantization} of $(M,\omega,\O,\eta)$ consists of a Hermitian vector bundle $E\to M$ and a faithful first-order Lie algebra representation
	\[
		Q:\O\to\End\,\Gamma(M,E),
	\]
	which preserves the inner product on the subspace of smooth $L^2$ sections of $\Gamma(M,E)$ with respect to $\eta$, and which extends the Hamiltonian vector fields on $M$ in the sense that 
	\[
		Q_f(s\psi) = (X_fs)\psi+sQ_f\psi,
	\]
	for all $f\in \mathcal{O}$, $s\in C^\infty(M)$, and $\psi\in\Gamma(E)$. By \emph{first-order}, we mean that if $f$ vanishes to first order at $x\in M$ then $(Q_f\psi)_x=0$ for all $\psi\in\mathcal{H}$. The operator $Q_f$ is the \emph{quantum observable} corresponding to $f\in\O$. We call $\mathcal{H}=\Gamma(M,E)$ the \emph{space of prequantum states}, and we say that $(M,\omega,\O,\eta)$ is \emph{quantizable} if it admits a prequantization.
\end{definition}

For any connection $\nabla'$ on $E$, we have $(Q-\nabla'_X)_f\, s\psi = s(Q-\nabla'_X)_f\psi$ for all $s\in C^\infty(M)$ and $f\in\O$, $s\in C^\infty(M)$, and $\psi\in\mathcal{H}$, so that $A'=Q-\nabla'_X$ is tensorial, and thus $Q_f=\nabla'_{X_f}+A'_f$ is a first-order differential operator on $E$ for every $f\in\O$. The Hamiltonian extension property of $Q$ is thus expressed diagrammatically as
\begin{center}
	\begin{tikzpicture}[scale=2]
		\node (A) at (-.2,0) {$\O$};
		\node (B) at (1,1) {$\mathcal{D}_1(M,E)_0$};
		\node (C) at (1,0) {$\mathfrak{X}(M)$,};
		
		\path[->,dashed] (A) edge node[above left] {$Q$} (B);
		\path[->] (A) edge node[below] {$X$} (C);
		\path[->] (B) edge node[right] {$\sigma$} (C);
	\end{tikzpicture}
\end{center}
where $\mathcal{D}_1(M,E)_0$ is the space of first-order operators $D$ on $E$ whose principal symbol $\sigma(D)$ involves only scalar multiplication on the fibers of $E$, so that in particular we may identify $\sigma(D)$ with a vector field on $M$. See, for example, \cite[Section 2.1]{BerlineGetzlerVergne92} for a reference on the symbol map $\sigma$ which corresponds with our use here.

Further observe that the Hamiltonian extension property is equivalent to the condition that $Q_f\hspace{.5pt}m_s = m_{X_fs}$, where $m_s$ denotes multiplication by $s\in C^\infty(M)$, and where the application of $Q_f$ on the operator $m_s$ is given by $(Q_f\hspace{.5pt}m_s)\hspace{.5pt}\psi = [Q_f,m_s]\psi$ for all $\psi\in\H$.

\begin{definition}
	A \emph{morphism of prequantizations} $\tilde\phi$ from a prequantization $(E,Q)$ of $(M,\omega,\O,\eta)$ to a prequantization $(E',Q')$ of $(M',\omega',\O',\eta')$ consists of a $V$-symplectomorphism $\phi:(M,\omega)\to (M',\omega')$ and a lift $\tilde\phi:E\to E'$ such that $\phi^*\O'\subset\O$ and $\tilde\phi^*(Q'_{f'}\,\psi') = Q_{\phi^*\!f'}\,\tilde\phi^*\psi'$, for all $f'\in\O'$ and $\psi'\in\mathcal{H}'$. We call $\tilde\phi$ an \emph{automorphism} when $(M,\omega,\O,\eta)=(M',\omega',\O',\eta')$ and when $\phi$ is the identity on $M$.
\end{definition}

Henceforth, we shall write $(M,\omega)$ for $(M,\omega,\O,\eta)$ and take $\O$ and $\eta$ to be understood.

Note that $E$ is naturally a bundle of $V$-modules, as the Hamiltonian extension property implies that $Q_v$ is tensorial for each $v\in V$, which we identify with the corresponding constant function on $M$. We will denote the fiber action by $A$, so that $A_v(\psi_x) = (Q_v\psi)_x$.

\begin{proposition}\label{prop:induced_V-linear_connection}
	If $(M,\omega)$ is transitive, then $\nabla_X\psi = (Q-A)_{f_X}\psi$ defines a $V$-linear connection on $E$, where $X\in\X_H(M)$ is a Hamiltonian vector field, and where $(A_f\psi)_x = A_{f(x)}\psi_x$ for all $x\in M$.
\end{proposition}

\begin{proof}
	Transitivity guarantees the existence of $f_X$. Since the difference of two Hamiltonian functions $f_X-f_X'$ is constant, $(Q-A)_{f_X-f_X'}=0$ and thus $\nabla$ is well-defined. If $X\in\X_H(M)$ extends the zero vector at a point $x\in M$, then $\d f_X$ vanishes at $x$. Since $Q$ is first-order, $\nabla_Xf=Q_{f_X-f_X(x)}\psi$ vanishes at $x$, so that $\nabla_X\psi$ is tensorial in $X$. The Leibniz property follows as
	\[
		(Q-A)_{f_X}(s\psi) = (X_{f_X}s)\psi+sQ_{f_X}\psi - sA_{f_X}\psi = (Xs)\psi + s(Q-A)_{f_X}\psi.
	\]
	Linearity and unitarity follow from the respective properties of $Q$ and $A$. Finally, since $Q$ is a morphism of Lie algebras,
	\[
		0 = Q_{\{f,v\}} = [\nabla_{X_f}+A_f,A_v] = [\nabla_{X_f},A_v],
	\]
	for all $f\in\O$ and $v\in V$, from which $[\nabla, A]=0$. That is, $\nabla$ preserves the $V$-module structure of $E$.
\end{proof}

We have shown that $Q_f=\nabla_{X_f}+A_f$ splits naturally into a first-order component $\nabla_{X_f}$ and a tensorial component $A_f$. Note that the $V$-linearity of $\nabla$ is equivalent to the property that $A$ is a parallel section of the bundle $\Hom(V,\End\,E)\to M$ with respect to the $\nabla$-induced connection.

\begin{lemma}\label{lem:V-module_structure_parallelism_condition}
	The $V$-module structure $A$ is parallel if and only if $[\nabla_{X_f},A_h] = A_{\{f,h\}}$.
\end{lemma}

\begin{proof}
	If $sv\in\O$ for $s\in C^\infty(M)$ and $v\in V$, then
	\[
		\nabla_{X_f} A_{sv} = \nabla_{X_f} sA_v = (X_fs)A_v + s\nabla_{X_f}A_v = A_{X_f(sv)} = A_{\{f,sv\}}.
	\]
	The forward implication follows by linearity in $h$. The converse follows as $\{\,\cdot\,,v\}=0$ for all $v\in V$.
\end{proof}

\begin{lemma}\label{prop:prequantum_curvature}
	The curvature $F^\nabla\in\Omega^2(M,\End\,E)$ is given by
	\[
		F^\nabla(X,Y)\,\psi = -\omega(X,Y)\cdot \psi,
	\]
	where $\cdot$ denotes the action of $V$ on $E$.
\end{lemma}

\begin{proof}
	Invoking Proposition \ref{prop:induced_V-linear_connection} and Lemma \ref{lem:V-module_structure_parallelism_condition}, we obtain
	\[
		[Q_f,Q_h] = [\nabla_{X_f}+A_f,\nabla_{X_h}+A_h] = [\nabla_{X_f},\nabla_{X_h}] + 2 A_{\{f,h\}},
	\]
	so that
	\begin{align*}
		\nabla_{X_{\{f,h\}}} + A_{\{f,h\}} = Q_{\{f,h\}} = [\nabla_{X_f},\nabla_{X_h}] + 2A_{\{f,h\}},
	\end{align*}
	and thus
 	\[
		-A_{\omega(X_f,X_h)} = [\nabla_{X_f},\nabla_{X_h}] - \nabla_{[X_f,X_h]} = F^\nabla(X_f,X_h).
	\]
	Here we have used the identities $\{f,h\}=\omega(X_f,X_h)$ and $X_{\{f,h\}}=[X_f,X_h]$.
\end{proof}

\begin{definition}
	A \emph{prequantum vector bundle} on $(M,\omega)$ consists of a faithful Hermitian $V$-module bundle $E\to M$ with a compatible unitary connection $\nabla$ satisfying $F^\nabla = -\omega$.
\end{definition}

By a \emph{$V$-module}, we mean a linear representation of the abelian Lie algebra $V$.

As with the case of prequantizations, we will denote the space of smooth sections $\Gamma(M,E)$ by $\mathcal{H}$ and the faithful $V$-module structure on $E$ by $A$.

\begin{theorem}
	If $(M,\omega)$ is transitive and connected, then there is a natural correspondence between prequantizations $(E,\nabla_X+A)$ and prequantum vector bundles $(E,\nabla,A)$ on $(M,\omega)$.
\end{theorem}

\begin{proof}
	We have shown in Proposition \ref{prop:induced_V-linear_connection} and Lemma \ref{prop:prequantum_curvature} that every prequantization $(E,Q)$ of $(M,\omega)$ satisfies $Q=\nabla_X+A$ for a $V$-module structure $A$ on the fibers of $E$ and a unitary $V$-linear connection $\nabla$ on $E$ with curvature $F^\nabla=-\omega$. It remains only to show that $A$ is faithful on fibers. Fix $x\in M$ and $v\in V$ and suppose $A_v$ vanishes at $x$. Since $A_v$ is parallel, it follows that $A_v$ vanishes on the connected component $M$ of $x$ so that $Q_v\psi = A_v\psi=0$ for all $\psi\in \mathcal{H}$. Thus, $v=0$ and we conclude that $A$ is faithful.

	Now suppose that $(E,\nabla,A)$ is a prequantum vector bundle. The curvature condition $\omega=-F^\nabla$ implies that
	\[
		A_{\{f,h\}} = -F^\nabla(X_f,X_h) = \nabla_{X_{\{f,h\}}} - [\nabla_{X_f},\nabla_{X_h}].
	\]
	An application of Lemma \ref{lem:V-module_structure_parallelism_condition} yields
	\begin{align*}
		[\nabla_{X_f}+A_f,\nabla_{X_h}+A_h]
			&=	[\nabla_{X_f},\nabla_{X_h}] + [\nabla_{X_f},A_h] + [A_f,\nabla_{X_h}] + [A_f,A_h]	\\
			&=	[\nabla_{X_f},\nabla_{X_h}] + 2A_{\{f,h\}}											\\
			&=	\nabla_{X_{\{f,h\}}} + A_{\{f,h\}}.
	\end{align*}
	Consequently, $Q_f=\nabla_{X_f}+A_f$ defines a representation of $\O$ on $\Gamma(M,E)$. If $f$ vanishes to first order at $x\in M$, then $X_f=0$ and $A_f=0$ at $x$, and thus $(Q_f\psi)_x=0$ for all $\psi\in\Gamma(M,E)$ so that $Q$ is first-order. The Hamiltonian extension property of $Q$ is an immediate consequence of the Leibniz property of $\nabla$ and the tensoriality of $A$. If $f\in\O$ is in the kernel of $Q$, then $\nabla_{X_f} = -A_f$ is tensorial, so that $X_f=0$, from which $A_f=0$, and thus $f=0$ and $Q$ is faithful. Therefore, $(E,\nabla_X+A)$ is a prequantization of $(M,\omega)$.
\end{proof}

\begin{remark}
	Under this correspondence, an automorphism of $(E,\nabla,A)$ is naturally identified with a parallel $V$-linear vector bundle automorphism $\tilde\phi_0\in\Aut\,E$. If $M$ is connected, then the parallelism of $A$ implies that $\tilde\phi_0$ is determined by its value at any point $x\in M$.
\end{remark}

Let $\lambda:V\to\C$ be an $\R$-linear map and recall that the \emph{$\lambda$-weight space} of the representation $A_x$ at the fiber $E_x$ is the subspace
\[
	(E_x)_\lambda = \{\sigma\in E_x \,|\, (A_x)_v\sigma = \lambda_v\sigma \text{ for all }v\in V\},
\]
The form $\lambda$ is said to be a \emph{weight} of $A_x$ if $(E_x)_\lambda$ is nonzero. We will denote the set of weights of $A_x$ by $w(A_x)$. Since the abelian Lie algebra action $A_x$ is unitary, $E_x$ splits as the sum of weight spaces and $w(A)\subset \i V^*$.

\begin{theorem}\label{thm:prequantization_structure}
	Every prequantum vector bundle $(E,\nabla,A)$ splits as the sum of weight bundles $\oplus_{\lambda\in w(A)}(E_\lambda,\nabla|_{E_\lambda},\lambda)$.
\end{theorem}

\begin{proof}
	Fix a point $x\in M$ and let $\tau_\gamma:E_x\to E_y$ denote the parallel transport along a path $\gamma:I\to M$ from $x$ to $y\in M$. Since $A_v$ is parallel, we have
	\[
		A_v(\tau_\gamma\sigma) = \tau_\gamma(A_v\sigma) = \lambda_v(\tau_\gamma \sigma),
	\]
	for every weight $\lambda:V\to\C$, $\lambda$-weight vector $\sigma\in E_x$, and $v\in V$. It follows that $\tau_\gamma$ establishes a bijection between the weight spaces $(E_x)_\lambda$ and $(E_y)_\lambda$. Thus, the weights $w(A)=w(A_x)$ are independent of $x\in M$. Furthermore, taking $x=y$ shows that $\mathrm{Hol}_x^\nabla$ preserves $(E_x)_\lambda$ and consequently that the parallel transport of $(E_x)_\lambda$ over $M$ describes a vector subbundle $E_\lambda$ with connection $\nabla|_{E_\lambda}$. Since parallel transport preserves weight spaces, the distribution of $V$-modules $(E_\lambda,\lambda)$ is the $V$-linear subbundle of $\lambda$-weight spaces of $(E,A)$. The result follows since the fiber $(E_x,A_x)$ splits as $\oplus_{\lambda\in w(A_x)}((E_x)_\lambda,\lambda)$ at $x$.
\end{proof}

As a consequence, the holonomy group $\mathrm{Hol}_x^\nabla\subset\End\,T_xM$ homomorphically embeds in the torus $U(1)^{\dim V}$.

\begin{remark}
	If each weight bundle $E_\lambda$ is a line bundle, then our approach to prequantization is essentially equivalent to a procedure for $k$-symplectic prequantization due to Awane and Goze \cite{AwaneGoze00}, in which a distinguished vector bundle $(L,\nabla)$ is defined on a $k$-symplectic manifold $(M,\omega_1,\ldots,\omega_k)$ to be the sum of the prequantum line bundles $(L_i,\nabla^{L_i})$ for each presymplectic manifold $(M,\omega_i)$ whenever the prequantum line bundles exists.
\end{remark}

\begin{corollary}
	Two prequantum vector bundles $(E,\nabla,A)$ and $(E',\nabla',A')$ on $(M,\omega)$ are equivalent if and only if $w(A)=w(A')$, where we consider $w(A)$ as a multiset in which the multiplicity of $\lambda\in w(A)$ is equal to its multiplicity as a weight of $A$.
\end{corollary}

\begin{proposition}\label{prop:prequantum_vector_bundle_rank}
	If $(E,\nabla,A)$ is a prequantum vector bundle on $(M,\omega)$, then $\mathrm{rank}\, E \geq \dim V$. In the case of equality, $w(A)$ is a basis of $\i V^*$.
\end{proposition}

\begin{proof}
	Since $E = \oplus_{\lambda\in w(A)}E_\lambda$, we have $\mathrm{rank}\,E\geq |w(A)|$. Observe that $w(A)$ spans $\i V^*$, since otherwise there is a $v\in V$ with $\lambda_v=0$ for each $\lambda\in w(A)$, so that $v\cdot E=0$ in violation of the faithfulness of $(E,A)$. Thus $|w(A)|\geq \dim V$.
\end{proof}

Note that we consider the complex rank of $E$ and the real dimension of $V$. We will say that $(E,\nabla_X+A)$ is a \emph{minimal rank prequantization} when $\mathrm{rank}\,E=\dim V$.


\subsection{Prequantum Lattices}

We turn now to the questions of existence and classification of prequantizations.

\begin{definition}
	We say that a full lattice $I\subset V$ is a \emph{prequantum lattice} for $(M,\omega)$ if $[\omega]_{H^2(M,V)}$ lies in the image of $H^2(M,I)\hookrightarrow H^2(M,V)$, that is, if the pairing $\langle \omega,\cdot \rangle:H_2(M,\Z)\to V$ takes values in $I$. We say that $I$ is \emph{principal} if it is minimal among prequantum lattices for $(M,\omega)$.
\end{definition}

Note that $I\subset V$ is a principal prequantum lattice precisely when $I$ is a sublattice of every prequantum lattice for $(M,\omega)$. We will denote by $I_\omega\subset V$ the image of the pairing $\langle\omega,\cdot\rangle:H_2(M,\Z)\to V$.

\begin{lemma}\label{lem:prequantum_lattice_condition}
	A lattice $I\subset V$ is a prequantum lattice for $(M,\omega)$ if and only if $I_\omega\subset I$. In this case, $I_\omega$ is a principal prequantum lattice for $(M,\omega)$.
\end{lemma}

\begin{proof}
	The first claim is a restatement of the condition for $I\subset V$ to be a prequantum lattice. If $I\subset V$ is a prequantum lattice, then since $I_\omega\subset I$ the additive subgroup $I_\omega\subset V$ is a lattice. The first claim implies that $I_\omega$ is a prequantum lattice and that $I_\omega$ contains every prequantum lattice for $(M,\omega)$.
\end{proof}

\begin{theorem}\label{thm:fundamental_theorem_of_prequantum_lattices}
	If $(M,\omega)$ is transitive and connected, then there is a bijection between equivalence classes of minimal rank prequantizations $(E,\nabla,A)$ on $(M,\omega)$ and bases $\mathcal{B}=\frac{1}{2\pi\i}w(A)$ of prequantum lattices $I\subset V$.
\end{theorem}

\begin{proof}
	If $(E,\nabla,A)$ is a minimal rank prequantum vector bundle on $(M,\omega)$, then Proposition \ref{prop:prequantum_vector_bundle_rank} provides that $\mathcal{B}^*=\frac{1}{2\pi\i}\,w(A)$ is a basis of $V^*$. For each $\lambda\in w(A)$, Theorem \ref{thm:prequantization_structure} implies that $\omega_\lambda=\lambda\circ\omega\in H^2(M,\C)$ represents the action of the curvature $F^{\nabla|_{E_\lambda}}\in\Omega^2(M,\End\,E_\lambda)$ of the $\lambda$-weight bundle $E_\lambda$, and thus $\langle \omega_\lambda, \Sigma\rangle \in 2\pi\i\,\Z$ for all $\Sigma\in H_2(M,\Z)$, so that $\langle w(A),I_\omega\rangle\subset 2\pi\i\,\Z$. If $I$ is the lattice generated by the dual basis $\mathcal{B}\subset V$, then it follows that $I_\omega$ is a sublattice of $I$, and consequently that $I$ is a prequantum lattice by Lemma \ref{lem:prequantum_lattice_condition}. Since $E=\oplus_{\lambda\in w(A)}E_\lambda$ has minimal rank, the factors $(E_\lambda,\nabla|_{E_\lambda})$ are necessarily line bundles $(L_\lambda,\nabla^{L_\lambda})$.
	
	Now suppose $\mathcal{B}$ is a basis of a prequantum lattice $I\subset V$. If $\lambda\in 2\pi\i\,\mathcal{B}^*$, then $I_\omega\subset I$ implies that $\langle\omega_\lambda,\Sigma\rangle\in 2\pi\i\,\Z$ for all $\Sigma\in H_2(M,\Z)$, and thus there is a line bundle $(L_\lambda,\nabla^{L_\lambda})\to M$ with curvature $F^{\nabla^{L_\lambda}}=-\omega_\lambda$. The sum $(E,\nabla,A)=\oplus_{\lambda\in 2\pi\i\,\mathcal{B}^*}(L_\lambda,\nabla^{L_\lambda},\lambda)$ is a minimal rank prequantum vector bundle on $(M,\omega)$ with $w(A)=2\pi\i\,\mathcal{B}^*$.
\end{proof}

\begin{remark}
	In the symplectic setting, we usually take $V=\R$, $I=\hbar\,\Z$, and $\mathcal{B}=\{\hbar\}$ for some $\hbar>0$. Note that the semiclassical limit $\hbar\to0$ is equivalent to $\hbar\,\Z\to\R$.
\end{remark}

\begin{corollary}
	If $(M,\omega)$ is transitive and connected, then the following are equivalent:
	\begin{enumerate}[i.]
		\item There is a minimal rank prequantization on $(M,\omega)$,
		\item $(M,\omega)$ admits a prequantum lattice $I\subset V$,
		\item $I_\omega$ is a principal prequantum lattice for $(M,\omega)$,
		\item $I_\omega$ is a lattice.
	\end{enumerate}
\end{corollary}

\begin{corollary}
	If $(M,\omega)$ is transitive and exact, then every lattice $I\subset V$ is a prequantum lattice for $(M,\omega)$. In particular, $(M,\omega)$ possesses a minimal rank prequantization.
\end{corollary}

\begin{proof}
	This follows as $I_\omega=0$ and by considering the connected components of $M$. Specifically, if $\omega=-\d\theta$, then a minimal rank prequantum vector is given by $E=M\times V^\C$, $\nabla_X = \mathcal{L}_X + \theta$, and any faithful action of $V$ on $V^\C$, where we equip $V^\C$ with a compatible Hermitian structure.
\end{proof}

\begin{example}\label{eg:Lie_group_prequantization}
	Consider a semisimple Lie group $G$ with its standard $\g$-symplectic structure $-\d\theta$, algebra of observables $\O\subset C^\infty(G,\g)\cong\X(G)$ given by the right-invariant vector fields on $G$, and Haar measure $\eta$. The complexified tangent bundle $T^\C G\cong G\times\g^\C$ equipped with the connection $\nabla$ induced by the Killing metric, and with the adjoint action $A_\xi X = \mathrm{ad}_\xi X$, is a prequantum vector bundle for $(G,-\d\theta)$.
\end{example}

\begin{example}\label{eg:polymomentum_prequantization}
	Fix a smooth manifold $Q$ and a vector space $V$, and let $\theta\in\Omega^1\big(\Hom(TQ,V),V\big)$ be given by $\theta(X) = \phi(\pi_*X)$ for $X\in T_\phi\Hom(TQ,V)$. Then $\omega=-\d\theta$ is a $V$-symplectic form on $\Hom(TQ,V)$. Since $\omega$ is exact, the trivial bundle $E=\Hom(TQ,V)\times V^\C$ is a prequantum vector bundle.
\end{example}

\subsection{Products}

Since the Lie algebra $V$ is abelian, every left $V$-module $U$ is naturally a right $V$-module and we may form the tensor product $U\otimes_V U'$. Explicitly, $v\cdot(u\otimes u')=vu\otimes u' = u\otimes vu'$.

Let $(E,\nabla,A)$ and $(E',\nabla',A')$ be prequantum vector bundles on $(M,\omega)$ and $(M',\omega')$, respectively.

\begin{definition}
	We define the \emph{product} of $(E,\nabla,A)$ and $(E',\nabla',A')$ to consist of the $V$-module bundle
	\[
		E\boxtimes_V E' = \pi^*E \otimes_V \pi'^*E' \longrightarrow M\times M',
	\]
	equipped with the connection $\nabla^{E\,\boxtimes_V E'} = \nabla^{\pi^*E}\otimes_V \nabla^{\pi'^*E'}$.
\end{definition}

\begin{lemma}\label{lem:tensor_weights}
	We have $w(A\boxtimes_V A') = w(A) \,\cap\, w(A')$.
\end{lemma}

\begin{proof}
	Since
	\[
		\lambda_v \;\sigma\otimes \sigma' =  v\cdot (\sigma\otimes\sigma') = \lambda'_v \;\sigma\otimes\sigma',
	\]
	for all $v\in V$, $\lambda\in w(A)$, $\lambda'\in w(A')$, $\sigma\in\pi^*(E_x)_\lambda$, and $\sigma'\in\pi'^*(E'_x)_{\lambda'}$, it follows that
	\[
		E\boxtimes_V E' = \bigoplus_{\substack{\lambda\in w(A) \\ \lambda'\in w(A')}} E_\lambda\boxtimes_V E'_{\lambda'}  = \bigoplus_{\lambda\in w(A)\cap w(A')} E_\lambda\boxtimes_V E'_\lambda.
	\]
\end{proof}

\begin{proposition}
	If $(E,\nabla,A)$ and $(E',\nabla',A)$ have minimal rank, then $E\boxtimes_V E'$ is a prequantum vector bundle on $(M\times M',\omega+\omega')$ if and only if $w(A)=w(A')$. In this case, $E\boxtimes_V E'$ has minimal rank.
\end{proposition}

\begin{proof}
	If $E\boxtimes_V E'$ is a prequantum vector bundle, then Lemma \ref{lem:tensor_weights} implies that $|w(A)\cap w(A')| = |w(A\boxtimes_V A')|\geq \dim V$. Since $|w(A)|=|w(A')|=\dim V$, it follows that $w(A)=w(A')$.
	
	Conversely, if $w(A)=w(A')$, then Lemma \ref{lem:tensor_weights} yields that $w(A\boxtimes_V A')=w(A)$ is a basis of $V^*$, from which it readily follows that the action $A\boxtimes_V A'$ is faithful, and consequently that $E\boxtimes_VE'$ is a prequantum vector bundle.
	
	In either case, $\dim E\boxtimes_VE'=|w(A)\cap w(A')|=\dim V$, so that $E\boxtimes_VE'$ has minimal rank.
\end{proof}

We obtain an important corollary.

\begin{corollary}
	The $k$-fold $V$-linear tensor bundle $E^{\otimes k}$ is a prequantum vector bundle for $(M,k\omega)$ for each integer $k>0$.
\end{corollary}

\begin{proof}
	This follows by identifying $(M,k\omega)$ as the diagonal of the $k$-fold product of $(M,\omega)$ with itself.
\end{proof}

\begin{definition}
	We call the collection $w(A)\subset V^*$ the \emph{type} of $(E,\nabla,A)$, and we say that $(E,\nabla_X+A)$ and $(E',\nabla'_X+A')$ are \emph{simultaneous prequantizations} when $w(A)=w(A')$. 
\end{definition}

\begin{remark}
	The type of a prequantization generalizes the constant $\hbar>0$ in the symplectic setting.
\end{remark}


\subsection{The Local Polysymplectic Case}

The definitions of prequantization and prequantum vector bundle extend naturally to the context of local $V$-symplectic manifolds. Fix a local $V$-symplectic manifold $(M,\omega)$ with bundle of coefficients $(\tilde{V},\nabla')$. Since $Q$ is first-order, there is a natural and well-defined action of $V_x$ on $E_x$ given by
\[
	v \cdot \psi_x = (Q_f \psi)_x
\]
for $v\in \tilde{V}_x$, $\psi\in\mathcal{H}$, and $f\in\O$ with $f(x)=v$ and $(\d^{\nabla'} f)_x =0$.

The results of Subsection \ref{subsec:construction_of_prequantizations} which precede Theorem \ref{thm:prequantization_structure} extend without modification to the local context. However, we have only a local splitting theorem for the prequantum vector bundle $(E,\nabla,A)$. In the absence of a fiber basis of parallel sections of $\tilde{V}\to M$, a role played by the constant functions in the global polysymplectic context, the factors of the splitting $E_x=\oplus_{\lambda\in w(A_x)}(E_x)_\lambda$ may be rearranged under parallel translation along nontrivial loops.

\begin{theorem}\label{cor:local_prequantization_structure}
	If $(M,\omega)$ is a transitive and connected local $V$-symplectic manifold with prequantum vector bundle $(E,\nabla,A)$, then the action of the holonomy group $\mathrm{Hol}_x^\nabla$ preserves the weight-space decomposition $E_x=\oplus_{\lambda\in w(A_x)} (E_x)_\lambda$ up to permutation of the factors.
\end{theorem}

\begin{proof}
	The proof is similar to that of Theorem \ref{thm:prequantization_structure}. Let $\gamma:I\to M$ be a loop based at $x$, $\lambda\in w(A_x)$ a weight of $A_x$, and $\sigma\in E_x$ a $\lambda$-weight vector. The parallelism of $A\in\Hom(\tilde{V},\End\,E)$ implies that
	\[
		A_v(\tau_\gamma\sigma) = \tau_\gamma(A_{\tau_\gamma^{-1}v}\sigma) = \lambda_{\tau_\gamma^{-1}v}(\tau_\gamma\sigma),
	\]
	so that $\tau_\gamma\sigma$ is a $(\tau_\gamma^{-1})^*\lambda$-weight vector for $A_x$. In particular, $\gamma$ induces a permutation of the weights $\lambda\in w(A_x)$ and a corresponding permutation of the weight spaces $(E_x)_\lambda$.
\end{proof}

\begin{corollary}
	The fundamental group $\pi_1(M,x)$ acts naturally on the set of weights $w(A_x)$.
\end{corollary}


\section{Polarized Quantization}\label{sec:polarized_quantization}

In this section, we address quantization with respect to real and complex polarizations, and conclude with an investigation of the interaction between quantization and reduction. 

Let $(M,\omega)$ be a $V$-symplectic manifold. Since our focus is on the geometry of prequantum vector bundles, we will not assume that $(M,\omega)$ is transitive, nor that it possesses an invariant measure.

\begin{definition}
	A \emph{polarization} of $(M,\omega)$ is an integrable Lagrangian distribution of the complexified tangent bundle $\mathcal{P}\subset T^\C\!M$. We say that $\mathcal{P}$ is \emph{real} when $\mathcal{P}=\overline{\mathcal{P}}$, and \emph{complex} when $\mathcal{P}\cap\overline{P}=0$.
\end{definition}

In the presence of a prequantum vector bundle $(E,\nabla,A)$ on $(M,\omega)$, we make the following definition.

\begin{definition}
	The \emph{space of $\mathcal{P}$-polarized quantum states} $\H_{\mathcal{P}}$ consists of those sections $\psi\in\H$ which are parallel along $\mathcal{P}$. That is, $\H_{\mathcal{P}}= \{\psi\in\H \,|\, \nabla_X\psi = 0 \text{ for all }X\in\mathcal{P}\}$.
\end{definition}

Throughout this section, we frequently extend $\R$-linear $V$-symplectic structures on a vector space $U$ to $\C$-linear $V$-symplectic structures on the complexification $U^\C=U\otimes_\R \C$ in the natural way without comment.


\subsection{Real Polarizations}

Let us briefly consider a real polarization $\mathcal{P}\subset T^\C M$ arising from Lagrangian foliations $\mathcal{F}$ of $(M,\omega)$. The space of quantum states $\mathcal{H}_{\mathcal{F}}$ is in bijective correspondence with the space of sections of $L/\mathcal{P}$ over the leaf space $M/\mathcal{F}$. We will denote $\H_{\mathcal{P}}$ by $\H_{\mathcal{F}}$.

\begin{example}
	Let $T$ be a maximal torus of the compact semisimple Lie group $G$ with Lie algebra $\t\subset\g$. Thus, $T$ is a Lagrangian submanifold of $G$. Since the left regular representation of $G$ on itself is $\g$-symplectic, it follows that $\mathcal{F}=(gT)_{g\in G}$ defines a Lagrangian foliation of $(G,-\d\theta)$, with associated polarization $G\cdot\t\subset TG$, and leaf space $T\backslash G$. It follows that $\mathcal{H}_{\mathcal{F}}$ is isomorphic to the space of $T$-invariant complex vector fields $\X^\C(G)_T$.
\end{example}

\begin{example}
	The fibers of $\Hom(TQ,V)\to Q$ define a Lagrangian foliation $\mathcal{F}$ of $\Hom(TQ,V)$ with leaf space isomorphic to $Q$. Thus, $\mathcal{H}_{\mathcal{F}}$ is naturally isomorphic to $C^\infty(Q,V^\C)$.
\end{example}

\begin{remark}
	We note that real polarizations have appeared previously in the context of the $k$-symplectic formalism \cite{Awane08,AwaneBanaliAmine18}.
\end{remark}


\subsection{Complex Polarizations}

Let us now equip $(M,\omega)$ with a compatible almost complex structure $J\in\End\,TM$. That is, $J^2=-1$ on fibers and $J^*\omega=\omega$.

\begin{definition}
	A linear complex structure $J\in\End\,U$ is said to be \emph{compatible} with $(U,\omega)$ when $J$ is a linear $V$-symplectomorphism of $(U,\omega)$, and we extend this language in the natural way to a $V$-symplectic manifold $(M,\omega)$ with almost complex structure $J\in\End\,TM$.
\end{definition}

\begin{lemma}\label{lem:complex_lagrangian_eigenspace_decomposition}
	The $\pm\i$-eigenspaces $U_\pm$ of $J$ are Lagrangian subspaces of $U^\C$.
\end{lemma}

\begin{proof}
	If $u,u'\in U_\pm$, then $\omega(u,u') = \omega(Ju,Ju') = -\omega(u,u')$, so $U_\pm\subset U_\pm^\omega$. It follows that both $U_+^\omega\cap U_-$ and $U_+\cap U_-^\omega$ are subspaces of $U_+^\omega\cap U_-^\omega = (U_++U_-)^\omega = 0$, and consequently that $U_\pm=U_\pm^\omega$.
\end{proof}

\begin{corollary}
	If $(E,\nabla,A)$ is a prequantum vector bundle on $(M,\omega)$, and if $J$ is a compatible complex structure on $(M,\omega)$, then
	\begin{enumerate}[i.]
		\item $T^{0,1}\!M$ is a polarization of $(M,\omega)$,
		\item $E$ is holomorphic and $\nabla$ is the Chern connection.
	\end{enumerate}
\end{corollary}

\begin{proof}
	\begin{enumerate}[i.]
		\item Lemma \ref{lem:complex_lagrangian_eigenspace_decomposition} implies that $T^{0,1}\!M$ is a Lagrangian distribution. Integrability follows by the Newlander-Nirenberg theorem.
		\item Since $\nabla^2=-A_\omega\in\Omega^2(M,\End\,E)$ vanishes when restricted to the Lagrangian distribution $T^{0,1}\!M$, it follows that $\nabla^{0,1}$ is a holomorphic structure on $E$.
	\end{enumerate}
\end{proof}

The $T^{0,1}M$-polarized sections of $\H$ are precisely the holomorphic sections of $(\H,\nabla)$. We will write $(M,\omega,J,G,\mu)$ to refer to a $V$-Hamiltonian system with $G$-invariant compatible complex structure, and we will denote $\H_{T^{0,1}M}$ by $\H_J$.

\begin{example}
	Every linear complex structure $J_U$ on a vector space $U$ determines a compatible linear complex structure $J = J_U\oplus(J_U^{-1})^*$ on the $V$-symplectic vector space $U\oplus\Hom(U,V)$. Since $J$ preserves the Lagrangian subspaces $U$ and $\Hom(U,V)$, it follows that $J$ is indefinite. Thus, every almost complex structure $J_Q$ on the smooth manifold $Q$ induces a compatible indefinite almost complex structure $J$ on $\Hom(TQ,V)$.
\end{example}

\begin{definition}\label{def:positive_definite}
	We will say that a linear complex structure $J$ on a $V$-symplectic vector space $(U,\omega)$ is \emph{definite} if the $V$-valued symmetric bilinear form $\langle u,u'\rangle=\omega(u,Ju')$ is definite.
\end{definition}

We apply this language in the natural way to almost complex $V$-symplectic manifolds $(M,\omega,J)$.

\begin{proposition}
	The linear complex structure $J$ is definite on $(U,\omega)$ if and only if the quadratic form $u\mapsto\omega(u,\bar{u})$ is definite on the $\pm i$-eigenspaces $U_\pm\subset U^\C$ of $J$.
\end{proposition}

\begin{proof}
	If $J$ is definite, then for every $u\in U$,
	\[
		\omega(Ju+\i u,Ju-\i u) = -2\i\,\omega(u,Ju) \neq 0.
	\]
	The forward implication follows since every member of $U_\pm$ is of the form $Ju\pm \i u$. The reverse implication is similar.
\end{proof}

\begin{definition}
	Let $(E,\nabla,A)$ be a prequantum vector bundle on the $V$-symplectic manifold $(M^{2n},\omega)$. We define the \emph{adapted volume} of $(M,\omega)$ to be
	\[
		\vol(M,\omega,A) = \frac{1}{(2\pi)^n n!} \sum_{\lambda\in w(A)} \int_M\omega_\lambda^n.
	\]
\end{definition}

\begin{definition}
	We will say that the prequantum vector bundle $(E,\nabla,A)$ on $(M,\omega,J)$ is \emph{definite} if $J$ is definite with respect to $F^\nabla=\omega$, \emph{positive} if $\lambda\,\langle X,X \rangle = \omega_\lambda(X,JX)\geq 0$ for each $X\in TM$ and $\lambda\in w(A)$, and \emph{strongly positive} if it splits as the sum of positive line bundles $\oplus_{\lambda\in w(A)}(E_\lambda,\nabla|_{E_\lambda})$.
\end{definition}

Using these definitions, we extend a well-known result from the symplectic setting.

\begin{proposition}\label{prop:volume_state_space_growth_rate}
	If $(E,\nabla,A)$ is a strongly positive prequantum vector bundle on $(M,\omega,J)$, then
	\[
		\dim \H_J(M,k\omega) = \vol(M,\omega,A)\,k^n + O(k^{n-1})
	\]
	as $k\to\infty$, where $\H_J(M,k\omega)$ denotes holomorphic sections of the tensor bundle $E^{\otimes k}\to M$.
\end{proposition}

\begin{proof}
	Since $E$ splits as the sum of line bundles $\oplus_{\lambda\in w(A)}E_\lambda$, the Riemann-Roch theorem implies that
	\begin{align*}
		\dim_{\Z_2} H^*(\partial+\bar\partial)
			&=	\frac{1}{(2\pi)^n}\int_M\mathrm{ch}\hspace{.5pt}E\wedge\mathrm{Td}\hspace{.5pt}M			\\
			&=	\frac{1}{(2\pi)^n}\sum_{\lambda\in w(A)}\int_M e^{k\omega_\lambda}\wedge\mathrm{Td}\hspace{.5pt}M	\\
			&=	\vol(M,k\omega,A) \;+\, O(k^{n-1}).
	\end{align*}
	Since $(E_\lambda,\nabla|_{E_\lambda})$ is positive for each $\lambda\in w(A)$, we have $\dim_{\Z_2} H^*(\partial+\bar\partial) = \dim \H_J(M,k\omega)$ for sufficiently large $k$ by an application of the Kodaira vanishing theorem \cite{Wells08}.
\end{proof}


\subsection{Quantization and Reduction}

We now turn to the interaction between quantization and $V$-Hamiltonian reduction. In particular, we will show that the original Guillemin-Sternberg conjecture does not extend to the polysymplectic setting.

If $(M,\omega,G,\mu)$ is a $V$-Hamiltonian system and $(E,\nabla,A)$ is a prequantum vector bundle on $(M,\omega)$, then there is an induced right action of $\g$ on $\mathcal{H}$, given by $\xi\mapsto Q_{\tilde\mu(\xi)}$. Diagrammatically, we have
\begin{center}
	\begin{tikzpicture}[scale=2.3]
		\node (ul) at (1,.8) {$C^\infty(M)$};
		\node (ur) at (2,.8) {$\End\,\mathcal{H}$};
		\node (ll) at (0,0) {$\g$};
		\node (lr) at (1,0) {$\X(M)$};
		
		\path[->,dashed] (ul) edge node[above] {$Q$} (ur);
		\path[->] (ul) edge node[left] {$X$} (lr);
		\path[->] (ur) edge node[below right] {$\sigma$} (lr);
		\path[->,dashed] (ll) edge node[above left] {$\tilde\mu$} (ul);
		\path[->] (ll) edge node[below] {$\lambda_*$} (lr);
	\end{tikzpicture}
\end{center}
where we note that $\sigma$, as defined in Subsection \ref{subsec:construction_of_prequantizations}, is only a partial map. We will assume that $Q_{\tilde\mu}$ determines a right action of $G$ on $\mathcal{H}$, as obtains, for example, when $G$ is compact and semisimple. We will also assume that the action of $G$ on $\mu^{-1}(0)$ is free and that the reduction $(M_0,\omega_0)$ is smooth.

In \cite{GuilleminSternberg82}, Guillemin and Sternberg established the following seminal result.

\begin{theorem}[K\"ahler Quantization Commutes with Reduction]\label{thm:GS_quantization_reduction}
	Let $(L,\nabla)$ be a positive prequantum line bundle on a K\"ahler manifold $(M,\omega,J)$, let $G$ be a compact connected Lie group acting on $(M,\omega,J)$ in a Hamiltonian fashion with moment map $\mu:M\to\g^*$, and let $\H_J(M)_G$ be the subspace of $G$-fixed members of $\H_J(M)$. Then,
	\begin{enumerate}[i.]
		\item The reduced space at $0\in\g^*$ is naturally a K\"ahler manifold $(M_0,\omega_0,J_0)$ with reduced prequantum line bundle $L_0$,
		\item There is a natural isomorphism $\H_J(M)_G \cong \H_{J_0}(M_0)$.
	\end{enumerate}
\end{theorem}

There was a significant amount of activity involved in generalizing this result, known as the \emph{Guillemin-Sternberg conjecture} or the \emph{$[Q,R]=0$ conjecture}, the latter notation expressing the commutation of quantization and reduction. It was first proved in a more general setting, in which $\H$ is defined to be the index space of a particular Dirac operator on the spin\textsuperscript{c} spinor bundle $\Lambda^{0,*} T^*M\otimes L$, by Vergne \cite{Vergne96a} in the case that $G$ is a torus, by Meinrenken \cite{Meinrenken96}, Paradan \cite{Paradan01}, and \cite{TianZhang98} in the general case, and by Meinrenken and Sjamaar \cite{MeinrenkenSjamaar99} for singular reduced spaces. The setting of these results agrees with K\"ahler quantization in the presence of a positive prequantum line bundle, so that these results extend those of Guillemin and Sternberg \cite{GuilleminSternberg82}. The $[Q,R]=0$ conjecture has also been proved in the context of spin\textsuperscript{c} quantization \cite{Paradan12,ParadanVergne17} and for certain presymplectic manifolds \cite{Cannas-da-SilvaKarshonTolman00,Hochs15}.

We will show that the Guillemin-Sternberg conjecture is false in the $V$-symplectic setting.

\begin{theorem}
	The presence of a positive definite prequantum vector bundle on $(M,\omega,J,G,\mu)$ does not imply that the reduced space $(M_0,\omega_0)$ inherits a complex structure $J_0$.
\end{theorem}

\begin{proof}
	Let $(L,\nabla)$ be a prequantum line bundle on the Hamiltonian system $(M,\omega,J,G,\mu)$, and suppose that $0\in\g^*$ is a regular value of $\mu$, that $M_0=\mu^{-1}(0)/G$ is a point, and that $G$ does not admit a complex structure. For example, we may take rotations of the complex sphere. Define the $\R^2$-Hamiltonian system $(M^2,\tilde\omega,J,G,\tilde\mu)$ so that $J$ is the product complex structure,
	\begin{align*}
		\tilde\omega(X_1+X_2,Y_1+Y_2)	&=	\omega(X_1,Y_1)\oplus\omega(X_2,Y_2)\;\in\R^2,	\hspace{1cm}X_i,Y_i\in T_{x_i}M, \;x_i\in M,	\\
		\tilde\mu(x_1,x_2)				&=	\mu(x_1)\oplus\mu(x_2),
	\end{align*}
	and $G$ acts simultaneously on each factor of $M^2$. Note that $\tilde{L}=\pi_1^*L + \pi_2^*L\to M^2$ is a positive definite prequantum vector bundle on $(M^2,\tilde\omega,J)$, that $0\in\Hom(\g,\R^2)$ is a regular value of $\tilde\mu$, and that $\tilde\mu^{-1}(0) = \mu^{-1}(0)\times\mu^{-1}(0)$. We conclude that the reduced space $(M^2)_0\cong G$ does not admit a complex structure.
\end{proof}

\begin{theorem}
	Let $(E,\nabla,A)$ be a positive definite prequantum vector bundle on $(M,\omega,J,G,\mu)$ and suppose that $(M_0,\omega_0)$ is nonempty and $V$-symplectic, and inherits a complex structure $J_0$ and prequantum vector bundle $E_0$. It is not generally the case that $\H_J(M)_G\cong\H_{J_0}(M_0)$.
\end{theorem}

\begin{proof}
	Let $(M,\omega_i,J)_{i=1,2}$ be positive K\"ahler manifolds with prequantum line bundles $(L_i,\nabla^{L_i})$, each equipped with the structure of a Hamiltonian system $(M,\omega_i,G,\mu_i)$ for a fixed action of $G$ in such a way that $M_1\cap M_2\subset M_1$ is a proper holomorphic inclusion of complex manifolds, for $M_i=\mu_i^{-1}(0)/G\subset M/G$. This condition is achieved, for example, by obtaining $(M,\omega_2,J)$ from $(M,\omega_1,J)$ via a smooth isotopy of $M$ preserving $J$ and the action of $G$. In the extreme case, $M_1\cap M_2$ is a point and the intersection is transverse. Put $\omega = \omega_1\oplus\omega_2$, $\mu=\mu_1\oplus\mu_2$, and observe that $(L_1\oplus L_2,\nabla^{L_1}+\nabla^{L_2})$ is a strongly positive prequantum vector bundle on the $\R^2$-Hamiltonian system $(M,\omega,G,\mu)$ with reduced space at $0\in\Hom(\g,\R^2)$ given by $M_0=M_1\cap M_2$. In particular, $M_0$ inherits a complex structure $J_0$ and the reduced vector bundle $L_0$ is K\"ahler in each factor and thus strongly positive. From Theorem \ref{thm:GS_quantization_reduction}, we have $\H_J(M)_G\cong\H_J(M,\omega_1)_G\oplus\H_J(M,\omega_2)_G\cong \H_{J_1}(M_1)\oplus\H_{J_2}(M_2)$, so that $\dim \H_J(M)_G\geq \dim \H_{J_1}(M_1)$. Since $\dim M_1 > \dim M_0$, Theorem \ref{prop:volume_state_space_growth_rate} yields $\dim \H_J(M,k\omega)_G > \dim \H_{J_0}(M_0,k\omega_0)$ for $k>0$ sufficiently large. In particular, $\H_J(M,k\omega)\not\cong\H_{J_0}(M_0,k\omega_0)$.
\end{proof}


\section{Spin\textsuperscript{c} Quantization}\label{sec:spinc_quantization}

In this section, we observe that there is a natural definition of spin\textsuperscript{c} quantization in the local polysymplectic setting.


\subsection{Review of Spin\textsuperscript{c} Dirac Operators}

Let us briefly review spin\textsuperscript{c} quantization in the usual symplectic context. We refer to \cite[Appendix D]{LawsonMichelsohn89} for spin\textsuperscript{c} geometry and to \cite{BerlineGetzlerVergne92,LawsonMichelsohn89,Friedrich00} for Clifford modules and Dirac operators.

The \emph{Clifford algebra} $C\ell(U)$ of a vector space $U$ with positive-definite inner product $\langle\,,\rangle$ is the quotient of the tensor algebra $\mathbf{T}(U)$ by the ideal $I$ generated by the relation $u\otimes u = -\langle u,u\rangle$. The \emph{Clifford bundle} $C\ell(M) = C\ell(T^*M)$ of a Riemannian manifold $M$ is the bundle of Clifford algebras associated to the cotangent fibers $T^*M$.

The \emph{spin group} $\Spin(n)$ is the simply connected double cover of the special orthogonal group $\mathrm{SO}(n)=\mathrm{SO}(\R^n)$, which we identify as a subgroup of $C\ell(n)$ by means of the following two constructions. First, under the linear isomorphism $c\ell(n)^\times\cong C\ell(n)$, where $c\ell(n)^\times$ is the Lie algebra of the group of units $C\ell(n)^\times$, the development of the Lie subalgebra $\R^n\subset C\ell(n)$ yields a subgroup $\mathrm{Pin}(n)$ which acts on $\R^n$ by the adjoint action. Second, as the ideal $I$ contains only even elements of the $\Z$-graded tensor algebra $\mathbf{T}(\R^n)$, the quotient $C\ell(n)$ inherits a $\Z_2$-grading, $C\ell(n)^\pm$. The restriction of the adjoint action of $\mathrm{Pin}(n)$ on $\R^n$ realizes the subgroup $\mathrm{Pin}(n)\cap C\ell(n)^+$ as the connected double cover $\Spin(n)$ of $\mathrm{SO}(n)$.

The \emph{spin\textsuperscript{c} group} is defined to be
\[
	\Spinc(n) = \Spin(n)\times_{\Z_2}\mathrm{U}(1),
\]
where $\Z_2$ acts on $\Spin(n)$ by deck transformations of the coving map $\Spin(n)\to\mathrm{SO}(n)$, and on $\mathrm{U}(1)$ by $\pm 1$. 
A \emph{spin\textsuperscript{c} structure} on a Riemannian manifold $M^n$ is a $\Spinc(n)$-equivariant double cover $P_{\Spinc(n)}\to \mathrm{SO}(M)\times P_{\mathrm{U}(1)}$, where $\Spinc(M)$ is a $\Spinc(n)$-principal bundle, $\mathrm{SO}(M)$ is the $\mathrm{SO}(n)$-principal bundle of orthonormal cotangent frames, and $P_{\mathrm{U}(1)}$ is a $\mathrm{U}(1)$-principal bundle on $M$. We will assume that $P_{\mathrm{U}(1)}$ is equipped with a principal connection so that, in conjunction with the Levi-Civita connection on $M$, there is an induced connection on $P_{\Spinc(n)}$.

From this point forward we specialize to the case in which $n$ is even. The Clifford algebra $C\ell(n)=C\ell(\R^n)$ has a distinguished complex representation $c:C\ell(n)\to\End_\C S(n)$ satisfying $C\ell(n)\cong \End_\C S(n)$, known as \emph{Clifford multiplication}. Combining this with the action $e^{\i t}\cdot \sigma = e^{2\i t} \sigma$ of $\mathrm{U}(1)$ on $S(n)$, we obtain the \emph{spinor bundle}
\[
	S(M) = C\ell(M) \times_{\Spinc(n)} S(n),
\]
a bundle of Clifford modules on the Riemannian manifold $M^n$. A consequence of the inclusion $\Spin(n)\subset C\ell(n)$ is that $S(n)$ is naturally a representation of $\Spinc(n)$, called the \emph{spinor representation}. It can be shown that $S(n)$ splits as the sum of two subrepresentations $S(n)^\pm$, called the \emph{positive} and \emph{negative half-spinor representations}, and we likewise obtain the \emph{half-spinor bundles} $S(M)^\pm$ on $M$.
The \emph{Dirac operator} $D:\Gamma(S)\to\Gamma(S)$ on the spinor bundle $S=S(M)$ is defined to be the composition $c\circ\nabla^S$ of the connection $\nabla^S:\Gamma(S)\to\Gamma(T^*M\otimes S)$, induced by the principal connection on $\Spinc(M)$, and the Clifford multiplication $c:\Gamma(T^*M\otimes S)\to \Gamma(S)$. The operator $D$ splits into positive and negative components $D^\pm:\Gamma(S^\pm)\to\Gamma(S^\mp)$. When $M$ is compact, the Dirac operator has finite-dimensional kernel and we define the \emph{index space} to be the $\Z_2$-graded vector space
\[
	\H_D = \ker D^+ \ominus \ker D^-.
\]


\subsection{Application to Quantization}

Consider a prequantum vector bundle $(E,\nabla,A)$ on a Riemannian local $V$-symplectic manifold $(M^{2n},\omega)$, and suppose that $P_{\Spinc(n)}\to\mathrm{SO}(M)\times\mathrm{U}(\det^2 E)$ is a spin\textsuperscript{c} structure on $M$, where $\mathrm{U}(\det^2 E)$ denotes the bundle of unitary frames on the square of the determinant bundle $(\det E)^2$. As above, the connection $\nabla$ on $E$ and the Levi-Civita connection on $M$ together yield a principal connection on $P_{\Spinc(n)}$, and thus a Dirac operator $D$ on a spinor bundle $S$ associated to $P_{\Spinc(n)}$.

In this situation, we make the following definition.

\begin{definition}
	The \emph{spin\textsuperscript{c} quantization} of $(M,\omega)$ with respect to $(E,\nabla,A)$ is defined to be the index space $\H_D=\ker D^+ \ominus\ker D^-$.
\end{definition}

If $(M,\omega)$ is globally $V$-symplectic and $E=\oplus_{\lambda\in w(A)} E_\lambda$ splits as the sum of line bundles, then $\det E \cong \otimes_{\lambda\in w(A)} E_\lambda$. In particular, when $(M,\omega)$ is a symplectic manifold and $(E,\nabla)$ is a prequantum line bundle, this definition coincides with the usual spin\textsuperscript{c} quantization \cite{Fuchs09,Hochs08}.

\begin{remark}
	In the symplectic setting, the spin\textsuperscript{c} structure employed in the results of Vergne \cite{Vergne96a}, Meinrenken \cite{Meinrenken96}, Tian-Zhang \cite{TianZhang98}, and Paradan \cite{Paradan01}, described following Theorem \ref{thm:GS_quantization_reduction}, is typically distinct to that of spin\textsuperscript{c} quantization. While the determinant line bundle of the spin\textsuperscript{c} spinor bundle $\Lambda^{0,*}T^*M\otimes L$ is isomorphic to $\kappa^*\otimes L^2$ \cite[Proposition D.50]{GuilleminGinzburgKarshon02}, where $\kappa^*$ is the anticanonical line bundle associated to the underlying almost complex structure $J$, the corresponding determinant line bundle is isomorphic to $L^2$ in the context of spin\textsuperscript{c} quantization. Indeed, the $[Q,R]=0$ conjecture constitutes a distinct question in each setting.

\end{remark}

We conclude this section by noting that the polysymplectic spin\textsuperscript{c} $[Q,R]=0$ conjecture remains open.


\section{Outlook}\label{sec:outlook}

We consider three directions in which the material may be developed or applied.

\begin{enumerate}[1.]
	\item \textbf{Chern-Simons theory}. Let $M$ be a smooth manifold of dimension at least $3$ and let $P$ be a $G$-principle bundle on $M$. In an earlier work \cite{Blacker19}, we show that the regular part of the moduli space of flat connections $\mathcal{M}(P)$ possesses a natural $H^2(M)$-valued presymplectic form $\omega_{\mathcal{M}(P)}$, obtained as the reduced $2$-form of a canonical $\Omega^2(M)/B^2(M)$-symplectic structure $\omega_{\mathcal{A}(P)}$ on the space of connections $\mathcal{A}(P)$. This generalizes the situation in which $M$ is a surface and $\omega_{\mathcal{M}(P)}$ is a symplectic structure. In the surface case, the prequantum line bundle on $(\M,\omega_{\mathcal{M}})$ constitutes the \emph{Chern-Simons line bundle}. The associated quantization theory has been the subject of much independent interest \cite{RamadasSingerWeitsman89,Krepski08,JeffreyWeitsman92,Charles16,ScheinostSchottenloher95}. We refer to \cite{Freed09,Witten89} for background on Chern-Simons theory more generally. It would be interesting to investigate the corresponding \emph{Chern-Simons vector bundle} on the moduli space $\mathcal{M}(P)$ in the case of a higher dimensional base space $M$, and to compare this approach with similar work in this direction \cite{Perez18,LopezPerez19}.

	\item \textbf{Multisymplectic geometry}. There is active interest in clarifying the status of symplectic constructions in the multisymplectic setting, particularly in respect to the notion of the moment map and reduction \cite{Echeverri-a-Enri-quezMunoz-LecandaRoman-Roy18,Herman18,Herman18a,RyvkinWurzbacherZambon18,RyvkinWurzbacher15}.
			
		See \cite{CantrijnIbortLeon99,RyvkinWurzbacher19} for general background on multisymplectic geometry. Methods of multisymplectic quantization have recently appeared by Rogers \cite{Rogers11,Rogers12,Rogers13}, Serajelahi \cite{Serajelahi15}, Barron and Seralejahi \cite{BarronSerajelahi17}, and others \cite{DeBellisSamannSzabo10,DeBellisSamannSzabo11,SamannSzabo12}. It is possible that there is an approach to multisymplectic quantization which resembles our method in Section \ref{sec:prequantization}. The question likewise stands for the more general construction of \emph{higher Dirac structures} \cite{BursztynMartinezAlbaRubio19}.

	\item \textbf{Quantum Field Theory}. Polysymplectic geometry is a natural framework for classical field theory \cite{Kanatchikov98,GiachettaMangiarottiSardanashvily99,RomanRoyNarciso09,Roman-RoyReySalgado11}, and in this context various approaches to quantization have been effected \cite{Kanatchikov04,Kanatchikov98a,Kanatchikov01,Kanatchikov15,Bashkirov04,Sardanashvily02,BashkirovSardanashvily04}. In the symplectic setting, following a \emph{metaplectic correction}, the predictions of polarized quantization are expected to conform with experimental results \cite{Woodhouse92}. It would be interesting to determine whether there exists an analogous modification by which our quantization formalism may be brought to describe actual physics systems.
\end{enumerate}

\bibliography{quantization}
\bibliographystyle{abbrv}

\par
\medskip
\begin{tabular}{@{}l@{}}
	\textsc{Department of Mathematics, East China Normal University,} \\
	\textsc{500 Dongchuan Road, Shanghai, 200241 P.R.\ China} \\[1.5pt]
	\textit{E-mail address}: \texttt{cblacker@math.ecnu.edu.cn}
\end{tabular}

\end{document}